\documentclass[10pt]{amsart}
\usepackage{amssymb,amsbsy,amsmath,amsfonts,amssymb}
\usepackage{latexsym,euscript,exscale}

\usepackage{times}

\title{Universally measurable subgroups of countable index}
\author {Christian Rosendal}
\address{Department of Mathematics, Statistics, and Computer Science (M/C 249)\\
University of Illinois at Chicago\\
851 S. Morgan St.\\
Chicago, IL 60607-7045\\
USA}
\email{rosendal@math.uic.edu}
\urladdr{http://www.math.uic.edu/$_~$rosendal}

\date {November 2008}
\newcommand {\N}{\mathbb N}

\newcommand{\om}{\omega}

\newcommand{\tom} {\emptyset}

\newcommand{\equi}{\Leftrightarrow}

\newcommand{\til}{\rightarrow}

\newcommand {\del}{ \; \big| \;}

\newcommand{\inv}{^{-1}}

\renewcommand {\a} {\forall}

\newtheorem{thm}{Theorem}

\newtheorem{lemme}[thm]{Lemma}

\newtheorem{prob}[thm]{Problem}

\begin{document}

\thanks{The author was partially supported by NSF grant DMS 0556368.}



\maketitle
\begin{abstract}
It is  proved that any countable index, universally measurable subgroup of a Polish group is open. By consequence, any universally measurable homomorphism from a Polish group into the infinite symmetric group $S_\infty$ is continuous. It is also shown that a universally measure homomorphism from a Polish group into a second countable, locally compact group is necessarily continuous.
\end{abstract}

The present work is motivated by an old problem of J.P.R. Christensen, which asks whether any universally measurable homomorphism between Polish groups is continuous. To fix the terminology, let us recall that a {\em Polish space} is a separable topological space whose topology can be induced by a complete metric. Also, a topological group is {\em Polish} in case its topology is Polish. A subset $A$ of a Polish space $X$ is said to be {\em Borel} if it belongs to the $\sigma$-algebra generated by the open sets and $A$ is {\em universally measurable} if it is measurable with respect to any Borel probability (or equivalently, $\sigma$-finite) measure on $X$, i.e., if for any Borel probability measure $\mu$ on $X$, $A$ differs from a Borel set by a set of $\mu$-measure $0$. The class of universally measurable sets clearly forms a $\sigma$-algebra, but contrary to the class of Borel sets, there is no ``algebraic'' procedure for generating the universally measurable sets from the open sets. Thus, the extent of this class is somewhat elusive and largely depends on additional set theoretical axioms. 

A function $f\colon X\til Y$ between Polish spaces $X$ and $Y$ is said to be {\em Borel measurable}, resp. {\em universally measurable}, if $f\inv (U)$ is Borel, resp. universally measurable, in $X$ for all open $U\subseteq Y$. A classical result due to H. Steinhaus and A. Weil states that if $G$ is a second countable, locally compact group (and hence Polish) and $H$ is a Polish group, then any universally measurable homomorphism $\pi\colon G\til H$ is continuous. Actually, for this it suffices that $\pi$ is measurable with respect to left  or right Haar measure on $G$.  However, this result relies heavily on the translation invariance of Haar measure, and, as any Polish group with a non-zero, quasi-invariant, $\sigma$-finite Borel measure is necessarily locally compact, the proof gives no indication whether the result should hold for general Polish $G$. Nevertheless, in the late 1960's, J.P.R. Christensen \cite{christensen} (see also \cite{christensen3}) introduced a notion of Haar null sets in more general Polish groups and was able to use this to prove an analogue of the Steinhaus--Weil result for Abelian Polish groups $G$. In fact, a consequence of Christensen's proof is that any universally measurable homomorphism from a Polish group $G$ into a Polish group $H$, where $H$ admits a compatible two-sided invariant metric, is continuous. In particular, this applies to the case when $H$ is Abelian, compact or countable discrete. It immediately follows from this that if $G$ is Polish and $N\leqslant G$ is a countable index,  universally measurable, {\em normal} subgroup, then $N$ is open in $G$. To see this, one just considers the quotient mapping $\pi\colon G\til G/N$, where $G/N$ is taken discrete.   Recently, S. Solecki \cite{solecki} was able to prove similar results without assumptions on $H$, but where, on the other hand, the domain group $G$ belongs to a class of Polish groups called {\em amenable at $1$}. However, the general question of whether any universally measurable homomorphism between Polish groups remains open. 

In this paper, we shall provide two extensions of Christensen's result by instead weakening the conditions on the range group $H$. Namely, we show that if $H$ either is second countable, locally compact or has a neighbourhood basis at $1$ consisting of open subgroups (equivalently, if $H$ embeds into $S_\infty$), then any universally measurable homomorphism $\pi\colon G\til H$ from a Polish group $G$  into $H$ is continuous. 

We recall that $S_\infty$ is the group of all permutations of $\N$, which we give the Polish topology whose basic open sets are
$$
\{g\in S_\infty\del g(n_i)=m_i, \; i\leqslant k\},
$$
where $n_1,\ldots,n_k, m_1,\ldots, m_k$ are natural numbers. Also, we say that a subset $A$ of a Polish group $G$ admits a {\em perfect set of disjoint left translates} if there is a non-empty perfect set $C\subseteq G$ such that for all $x\neq y\in C$ we have $xA\cap yA=\tom$.  Similarly, we say that a subgroup $K$ of $G$ has {\em perfect index} if it admits a perfect set of disjoint left translates (in particular, $[G\colon K]=2^{\aleph_0}$). Finally, a universally measurable subset $A$ of a Polish group $G$ is said to be {\em right Haar null} if there is a Borel probability measure $\mu$ on $G$ such that for all $x\in G$ we have $\mu(Ax)=0$. Similarly for right translates and left Haar null.
With this we can now state our main results.

\begin{thm}\label{main}
Let $G$ be a Polish group and $F\leqslant G$ a universally measurable subgroup. Then $F$ either is open in $G$ or has perfect index. By consequence, any universally measurable homomorphism from a Polish group into $S_\infty$ is continuous.
\end{thm}

\begin{thm}\label{main2}
Let $\pi\colon G\til H$ be a universally measurable homomorphism from a Polish group $G$ into a second countable, locally compact group $H$. Then $\pi$ is continuous.
\end{thm}

We begin by first considering the connection between right Haar null sets and the number of disjoint translates. For this we need the following result.

\begin{thm}[J. Mycielski \cite{mycielski}]\label{myc}
Let $X$ be an uncountable Polish space equipped with a non-zero,  $\sigma$-finite, continuous Borel
measure $\mu$ and suppose that $R\subseteq X^2$ is a binary relation on $X$ of
$\mu\times \mu$-measure $0$. Then there is a Cantor set $C\subseteq X$ such
that
$$
\a x\neq y\in C\;\; \;\;\;(x,y)\notin R.
$$
\end{thm}

In attempting to prove results of automatic continuity, the passage from $A$ to $AA\inv$ is
often without consequence, so the following consequence of Mycielski's Theorem could be a useful
reformulation of being right Haar null.
\begin{lemme}\label{mycielski}
Suppose $G$ is a Polish group and $A,B\subseteq G$ are universally measurable.
\begin{enumerate}
  \item If $A$ admits a perfect set of disjoint left translates, then $A$ is right  Haar null.
  \item If $AA\inv \subseteq B$ and $B$ is right Haar null, then $A$ has a
  perfect set of disjoint left translates.
\end{enumerate}
\end{lemme}

\begin{proof}
(1) Suppose $P\subseteq G$ is a non-empty perfect set such that for distinct $x,y\in P$, we have $xA\cap yA=\tom$. Then for any $z\in G$, as 
$$
x\inv, y\inv\in Az\;\equi\; z\inv \in xA\cap yA,
$$
we see that $Az$ intersects $P\inv$ in at most one point. Therefore, if $\mu$ is any continuous Borel probability measure supported on $P\inv$, we have $\mu(Az)=0$ for all $z\in G$, which shows that $A$ is right Haar null.

(2) On the other hand, if $AA\inv \subseteq B$ and $B$ is right Haar null, as witnessed by a Borel probability measure $\mu$ on $G$, we see that
$$
\mu(By)=0,\quad \textrm{ for $\mu$-a.e. $y\in G$},
$$
and so by Fubini's Theorem
$$
\mu\times \mu(\{(x,y)\in G^2\del xy\inv\in B\})=0.
$$
Therefore, by Theorem \ref{myc}, there is a Cantor set $C\subseteq G$ such that for distinct $x,y\in C$, we have $xy\inv \notin B\supseteq A A\inv$, i.e., $x\inv A\cap y\inv A=\tom$. Setting $P=C\inv$, we see that $A$ has a perfect set of disjoint left translates.
\end{proof}
One should note the use of the external set $B$ in (2) above. This is needed since there is in general no reason to conclude that $A\inv A$ is universally measurable from the fact that $A$ is universally measurable.

Now, for the proof of our theorems we shall need two results.
\begin{thm}[J.P.R. Christensen \cite{christensen}]\label{christensen}
Suppose $G$ is a Polish group and $A_n\subseteq G$ are universally measurable subsets covering $G$, i.e., $G=\bigcup_{n\in \N}A_n$. 
Then there are $n$ and elements $h_1,\ldots,h_m\in G$ such that
$$
h_1A_nA_n\inv h_1\inv\cup\ldots\cup h_m A_nA_n\inv h_m\inv
$$
is a neighbourhood of the identity in $G$.
\end{thm}
The second result we need is proved by just slightly amending Christensen's proof of Theorem \ref{christensen} (see \cite{automatic} for a complete proof of the exact statement below).

\begin{thm}\label{christensen2}
Suppose $G$ is a Polish group and $A\subseteq G$ is a universally measurable
subset which is not right Haar null. Then for any open $W\ni 1$ there are finitely many
$h_1,\ldots,h_n\in W$ such that
$$
h_1AA\inv h_1\inv\cup\ldots\cup h_n AA\inv h_n\inv
$$
is a neighbourhood of the identity.
\end{thm}

We now come to the proof of Theorem \ref{main}.

\begin{proof}
Suppose $F\leqslant G$ is a universally measurable subgroup and that $F$ does not have perfect index in $G$.
Then, by Lemma \ref{mycielski}, $F$ cannot be right Haar null. So, by Theorem \ref{christensen2}, there are $g_1,\ldots,g_n\in G$ such that 
$$
V=g_1\inv Fg_1 \cup\ldots\cup g_n\inv Fg_n
$$ 
is a neighbourhood of $1$ in $G$, and, in particular, there is a countable set $D\subseteq G$ such that
$$
G=DV=DFg_1\cup\ldots\cup DFg_n.
$$
We claim that  $F$ has countable index in $G$. For if $G\neq DF=DFg_1g_1\inv F\inv$, pick some $a\in G\setminus DFg_1g_1\inv F\inv$, whereby $aFg_1\cap DFg_1=\tom$, and so
$$
aFg_1\subseteq DFg_2\cup\ldots \cup DFg_n
$$
and
$$
G=(Da\inv D\cup D)Fg_2\cup\ldots \cup (Da\inv D\cup D)Fg_n.
$$
Continuing this way, we either see that $F$ covers $G$ by a countable number of left translates or that
$$
G=D'Fg_n
$$
for some countable set $D'$, whereby also $G=D'F$. In any case, $F$ has countable index in $G$.

Now let the $g_i$  and $V$ be chosen as above. Then 
$$
K=g_1\inv Fg_1 \cap\ldots\cap g_n\inv Fg_n
$$
is a subgroup of countable index in $G$, whereby $\overline K$ is a closed subgroup of countable index and hence is open in $G$ by the Baire Category Theorem.
We then have
\begin{displaymath}\begin{split}
\overline K\subseteq VK&= (g_1\inv Fg_1 \cup\ldots\cup g_n\inv Fg_n)(g_1\inv Fg_1 \cap\ldots\cap g_n\inv Fg_n)\\
&\subseteq g_1\inv Fg_1 \cup\ldots\cup g_n\inv Fg_n .
\end{split}\end{displaymath}
By Neumann's Lemma this implies that there is some $i$ such that $L=g_i\inv Fg_i\cap \overline K$ has finite index in $\overline K$. By a lemma of Poincar\'e, it follows that there are finitely many $k_1,\ldots, k_m\in \overline K$ such that $M=\bigcap_{i=1}^mk_iLk_i\inv$ is a finite index, normal subgroup of $\overline K$. So, as it is also universally measurable, it is open in $\overline K$ by Christensen's Theorem. Since $\overline K$ is itself  open in $G$ this shows that $M$, $g_i\inv Fg_i$ and thus also $F$ are open in $G$.

Now suppose $\pi\colon G\til S_\infty$ is a universally measurable homomorphism. Then if $V\ni 1$ is a neighbourhood of the identity in $S_\infty$, there is an open subgroup $H\leqslant S_\infty$ of countable index such that $H\subseteq V$. But then $U=\pi\inv(H)$ is a universally measurable, countable index subgroup of $G$ and hence open. Since $\pi(U)\subseteq V$, this shows that $\pi$ is continuous at $1$ and therefore continuous everywhere. 
\end{proof}

Using this, we can extract some additional information.

\begin{thm}
Let $A$ be a universally measurable, symmetric subset of a Polish group $G$ containing $1$ and covering $G$ by countably many left translates. Then there is a $k$ such that
$A^k$ is a neighbourhood of $1$.
\end{thm}

\begin{proof}
Let $H=A^{<\om}$ be the set of all finite products of elements in $A$ and notice that $H$ is a subgroup of $G$ of countable index. Notice also that as $A\subseteq H$, we have 
$$
gA\cap H\neq \tom \equi g\in H.
$$
So if $G=\bigcup_{n\in \N}g_nA$, we see that $H=\bigcup_{g_n\in H}g_nA$, i.e., $H$ is a countable union of universally measurable sets, and therefore universally measurable itself. So, by Theorem \ref{main}, $H$ is open in $G$. Using Theorem \ref{christensen}, we find $h_1,\ldots, h_m\in H$ such that
$$
h_1\inv A^2h_1\cup\ldots\cup h_m\inv A^2h_m
$$
is a neighbourhood of $1$ and letting $k$ be large enough such that $h_i\in A^k$ for all $i$, we see that $A^{2k+2}$ is a neighbourhood of $1$.
\end{proof}

The above result is short of solving Christensen's problem by only an alternation of quantifiers. That is, if for any Polish group $G$ one could find a $k$ that works simultaneously for all $A\subseteq G$, then this would suffice to prove that any universally measurable homomorphism between Polish groups is continuous. However, the techniques presented here do not seem to suffice for this.

Instead, one could make an attempt at other test questions. One case where the group theoretical approach used above might be successful is for homomorphisms defined on $S_\infty$.
 \begin{prob}
 Suppose $G$ is a closed subgroup of $S_\infty$ and 
 $$
 \pi\colon G\til H
 $$
 is a universally measurable homomorphism from $G$ into a Polish group $H$. Is $\pi$ continuous?
 \end{prob}
 We should mention that for $G=S_\infty$ and many other specific closed subgroups the answer is positive, in fact, any homomorphism (universally measurable or not) from $S_\infty$ into a Polish group is continuous (see \cite{turbulence}). So the question is really only interesting when all closed subgroups $G$ are considered.

\

Our second result, Theorem \ref{main2}, which we will prove now, deals with locally compact groups and will rely on both Theorems \ref{christensen} and \ref{christensen2}.

\begin{proof}
Suppose $G$ is Polish, $H$ is second countable, locally compact and $\pi\colon G\til H$ is a universally measurable homomorphism. Since $H$ is second countable, we can find symmetric compact sets 
$$
K_0\subseteq K_1\subseteq K_2\subseteq \ldots \subseteq H=\bigcup_nK_n
$$ 
covering $H$ and such that $K_nK_n\subseteq K_{n+1}$ for all $n$.

Set $A_n=\pi\inv(K_n)$. Then 
$$
A_0\subseteq A_1\subseteq A_2\subseteq\ldots \subseteq G=\bigcup_nA_n
$$ 
is a chain of universally measurable subsets covering $G$. So, by Theorem \ref{christensen}, there is some $n$ and $g_1,\ldots,g_n\in G$ such that 
$$
g_1A_nA_n\inv g_1\inv \cup \ldots \cup g_nA_nA_n \inv g_n\inv
$$
is a neighbourhood of $1$ in $G$.
Let $m\geqslant n$ be large enough such that $g_1,\ldots,g_n\in A_m$. Then, 
$$
g_1A_nA_n\inv g_1\inv \cup \ldots \cup g_nA_nA_n \inv g_n\inv\subseteq A_mA_mA_m\inv A_m\inv \subseteq A_{m+2}
$$
and so $A_{m+2}$ is a neighbourhood of $1$.

To see that $\pi$ is continuous, it suffices to prove continuity at $1\in G$. So suppose $V$ is any neighbourhood of $1\in H$ and find some smaller open neighbourhood $U\ni 1$ such that for any $h\in K_{m+2}$, 
$$
hUU\inv h\inv\subseteq V.
$$ 
This is possible, since $K_{m+2}$ is compact.

Now, $\pi\inv(U)$ is universally measurable and covers $G$ by countably many right translates, so    $\pi\inv(U)$ is not right Haar null. Thus, by Theorem \ref{christensen2}, there are finitely many $f_1,\ldots,f_p\in A_{m+2}$ so that 
$$
W=f_1\pi\inv(U)\pi\inv(U)\inv f_1\inv\cup\ldots\cup f_p\pi\inv(U)\pi\inv(U)\inv f_p\inv
$$
is a neighbourhood of $1$ in $G$. But $\pi(f_i)\in K_{m+2}$, so
$$
\pi(W)\subseteq \pi(f_1)UU\inv \pi(f_1)\inv\cup\ldots\cup \pi(f_p)UU\inv \pi(f_p)\inv\subseteq V.
$$
Since $V$ is arbitrary, this shows that $\pi$ is continuous at $1$. 
\end{proof}

\end{document}